\documentclass[11pt]{amsart}
\usepackage{amsmath}
\usepackage{amsfonts}
\usepackage{amsthm}
\usepackage{amssymb}
\usepackage{mathrsfs}
\usepackage{latexsym}
\usepackage{verbatim}
\usepackage{diagrams}
\usepackage{dsfont}
\usepackage{stmaryrd}
\usepackage{graphics}

\swapnumbers

\newtheorem{theorem}{Theorem}[section]

\newtheorem{df}[theorem]{Definition}
\newtheorem{lemma}[theorem]{Lemma}
\newtheorem{prop}[theorem]{Proposition}

\newtheorem{corr}[theorem]{Corollary}

\newtheorem{example}[theorem]{Example}
\newtheorem{remark}[theorem]{Remark}

\newarrow{Equals}{=}{=}{=}{=}{=}

\def\residual{\mathrm{res}}

\def\etale{\'{e}tale }

\def\vrho{\varrho}
\def\vrhobar{\varrhobar}
\def\rbar{\overline{r}}
\def\rhobar{\overline{\rho}}
\def\vrhobar{\overline{\varrho}}

\def\PGL{\mathrm{PGL}}

\def\Sym{\mathrm{Sym}}
\def\PSL{\mathrm{PSL}}
\def\SL{\mathrm{SL}}

\def\F{\mathbf{F}}
\def\Fbar{\overline{\F}}
\def\mubar{\overline{\mu}}
\def\Q{\mathbf{Q}}
\def\GL{\mathrm{GL}}
\def\p{\mathfrak{p}}
\def\q{\mathfrak{q}}
\def\Z{\mathbf{Z}}
\def\Gal{\mathrm{Gal}}
\def\eps{\epsilon}
\def\epsp{\omega}
\def\ord{\mathrm{ord}}
\def\OL{\mathcal{O}}
\def\Qbar{\overline{\Q}}

\def\X{\mathfrak{X}}
\def\Xgeom{\X^{\mathrm{geom}}}
\def\v{\mathbf{v}}
\def\w{\mathbf{w}}
\def\Rbox{R^{\Box}}

\def\Rboxtv{R^{\Box,\tau,\v}}

\def\Spec{\mathrm{Spec}}

\def\LL{\mathscr{L}}
\def\C{\mathbf{C}}

\def\lifting{2.2.1 }
\def\liftingflat{2.2.1 }
\def\lifta{1.3 }
\def\liftb{1.4 }
\def\liftc{1.3.2 }
\def\zap{\kern+0.05em{\text{\reflectbox{$\lightning$} }}}

\def\t{r}
\def\Kbar{\overline{K}}
\def\Sh{\mathrm{Sh}}
\def\A{\mathcal{A}}

\begin{document}

\title{Even Galois Representations and the Fontaine--Mazur conjecture II}
\author{Frank Calegari}
\thanks{Supported in part by 
NSF Career Grant DMS-0846285 
and the Sloan Foundation.
}
\begin{abstract} We prove, under mild hypotheses, that
there are no irreducible  two-dimensional potentially semi-stable
\emph{even} $p$-adic Galois representations
of $\Gal(\Qbar/\Q)$ with distinct Hodge--Tate weights. This removes the
ordinary hypotheses required in our previous work~\cite{CalegariFM}.
We construct examples of irreducible two-dimensional residual 
representations that have no characteristic zero geometric
deformations.
\end{abstract}
\maketitle

\section{Introduction}

Let $G_{\Q}$ denote the absolute Galois group of $\Q$, and let
$$\rho: G_{\Q} \rightarrow \GL_2(\Qbar_p)$$
be a continuous irreducible representation unramified away from finitely many primes.
In~\cite{FM}, Fontaine and Mazur conjecture that if $\rho$ is semi-stable at $p$,
then either $\rho$ is the Tate twist of an even representation with finite image or
$\rho$ is modular. In~\cite{KisinFM}, Kisin establishes this conjecture  in almost all
cases under the additional hypotheses that $\rho|D_p$ has distinct Hodge--Tate weights and
$\rho$ is \emph{odd} (see also~\cite{EmertonFM}). The oddness condition in Kisin's work is required
 in order to invoke
the work of Khare and Wintenberger~\cite{Khare, Khare2} on Serre's conjecture. If $\rho$ is even
and $p > 2$, however, then $\rhobar$ will never be modular. Indeed, when $\rho$ is
even and $\rho|D_p$ has distinct Hodge--Tate weights, the conjecture of Fontaine
and Mazur predicts that $\rho$ does not exist.
In~\cite{CalegariFM}, some progress was made towards proving this claim  under the additional assumption
that $\rho$ was \emph{ordinary} at $p$. The main result of this paper is to remove this
condition. Up to conjugation, the image of $\rho$ lands in $\GL_2(\OL)$ where
$\OL$ is the ring of integers of some finite extension $L/\Q_p$ (see Lemme~2.2.1.1 of~\cite{BM}). Let $\F$ denote the residue field, and let
$\rhobar: G_{\Q} \rightarrow \GL_2(\F)$ denote the corresponding residual
representation. We prove:

\begin{theorem}  Let
$\rho: G_{\Q} \rightarrow \GL_2(\Qbar_p)$ be a continuous \label{theorem:even}
Galois representation which is unramified except at a finite number of primes. Suppose
that $p > 7$, and, furthermore, that
\begin{enumerate}
\item $\rho |D_p$ is potentially semi-stable, with distinct Hodge--Tate weights.
\item The residual representation $\rhobar$ 
is absolutely irreducible
and not of dihedral type.
\item $\rhobar|D_p$ is not a twist of a representation of the form
$\displaystyle{ \left( \begin{matrix}\epsp & * \\ 0 & 1  \end{matrix} \right)}$ where
$\epsp$ is the mod-$p$ cyclotomic character.
\end{enumerate}
Then $\rho$ is modular.
\end{theorem}

Taking into account the work of Colmez~\cite{Colmez} and Emerton~\cite{EmertonFM}, this follows
directly from the main result of Kisin~\cite{KisinFM} when $\rho$ is odd.
Thus, it suffices to assume that $\rho$ is even and derive a contradiction.
As in~\cite{CalegariFM}, the main idea is to use \emph{potential automorphy}
to construct from $\rho$ a  RAESDC automorphic representation $\pi$ for $\GL(n)$ over some totally real
field $F$ whose existence is incompatible with the evenness of $\rho$.
It was noted in~\cite{CalegariFM} that improved automorphy lifting theorems would lead
to an improvement in the main results of that paper.  
Using the recent work of Barnet--Lamb, Gee, Geraghty, and
Taylor~\cite{BLGGT},  it is a simple
matter to deduce the main theorem of this paper if $\rho$ is a twist of a crystalline representation 
sufficiently deep in the 
Fontaine--Laffaille range (explicitly, if \emph{twice} the difference of the Hodge--Tate weights is
at most $p - 2$).  However, if one wants to apply the main automorphy lifting theorem
(Theorem~4.2.1) of~\cite{BLGGT} more generally,
then (at the very least) one has to assume that $\rho|D_p$ is potentially crystalline. Even under this
assumption,
 one runs into the difficulty of showing that
$\rho|D_p$ is \emph{potentially diagonalizable} (in the notation of that paper) which seems out of reach
at present. Instead, we use an idea we learnt from Gee (which is also 
crucially used in~\cite{BLGGT, BLGG,BLGGtwo}) of tensoring
together certain ``shadow'' representations in order to manoeuvre ourselves into a situation
in which we can show a certain representation (which we would like to prove is automorphic)
 lies on the same component (of a particular local deformation ring) as an automorphic representation.
 In~\cite{BLGGT,BLGG}, it is important that one restricts,
 following the idea of M.~Harris,  to tensoring with representations induced from
 characters, since then one is still able to prove the modularity of the original representation.
 In contrast, we shall need to tensor together representations with large image. Ultimately, we construct
 (from $\rho$) a regular algebraic self dual automorphic representation for $\GL(9)$ over a totally real field $E^{+}$
 with a corresponding $p$-adic Galois representation $\vrho: G_{E^{+}} \rightarrow \GL_9(\Qbar_p)$.
 If $\rho$ is even, then (by construction) it will the case that
 $\mathrm{Trace}(\vrho(c)) = +3$ for any complex conjugation $c$.
 This contradicts the main theorem of~\cite{TP}, and thus $\rho$ must be odd.
  In order to understand the local deformation rings that arise,  and in order to construct an
appropriate shadow representation, 
we shall have to use the full strength of the results of Kisin~\cite{KisinFM} for
totally real fields in which $p$ splits completely. This is the reason why condition~$3$ of
Theorem~\ref{theorem:even} is required, even when $\rho$ is even.

\medskip

It will be convenient to prove the following,  which, in light of
the main theorem of~\cite{KisinFM}, implies Theorem~\ref{theorem:even}.
Recall that $\epsp$ denotes the mod-$p$ cyclotomic character.

\begin{theorem} Let $F^{+}$ be a totally real  \label{theorem:even2}
field in which $p$ splits completely.
Let $\rho: G_{F^+} \rightarrow \GL_2(\Qbar_p)$ be a continuous  Galois
representation unramified except at a finite number of primes. Suppose that $p > 7$, and,
furthermore, that
\begin{enumerate}
\item $\rho |D_v$ is potentially semi-stable, with distinct Hodge--Tate weights, for all $v|p$.
\item The representation $\Sym^2 \rhobar |_{G_{F^{+}(\zeta_p)}}$ is irreducible.
\item If $v|p$, then $\rhobar|D_v$  is  
not a twist of a representation of the form
$\displaystyle{ \left( \begin{matrix}\epsp & * \\ 0 & 1  \end{matrix} \right)}$.
\end{enumerate}
Then, for every real place of $F^{+}$, $\rho$ is odd. 
\end{theorem}

\begin{remark}  \emph{Under the conditions of Theorem~\ref{theorem:even2}, 
it follows that $\rho$ is potentially modular over an extension in which $p$
splits completely (see Remark~\ref{remark:oddcase}).}
\end{remark}

\medskip

In section~\ref{section:even}, we give some applications of our theorem
to universal deformation rings. In particular, we construct (unrestricted) universal deformation
rings of large dimension such that none of the corresponding Galois
representations are geometric ($=$ de Rham). For an even more concrete application, we prove the following (See Corollary~\ref{corr:better}):

\begin{theorem}
There exists a surjective even  representation $\rhobar: G_{\Q} \rightarrow \SL_2(\F_{11})$
with no geometric deformations.
\end{theorem}

Contrast this result with Corollary~1(a) of Ramakrishna~\cite{R}, which implies that $\rhobar$
\emph{does} admit a deformation to a surjective representation 
$\rhobar: G_{\Q} \rightarrow \SL_2(\Z_{11})$ unramified outside finitely many primes.

\medskip

\begin{remark} A word on notation. \emph{There are only finitely many 
letters that can plausibly be used to denote a global field, and thus, throughout the text, we
have resorted to using subscripts. In order to prepare the reader, we note now the
existence in the text of a sequence of inclusions of totally real fields:
$$F^{+} \subseteq F^{+}_1 \subseteq F^{+}_2 \subseteq F^{+}_3 \subseteq
F^{+}_4 \subseteq F^{+}_5 \subseteq F^{+}_6 \subseteq F^{+}_7,$$
and corresponding degree two CM extensions $F^{}_3 \subseteq  \ldots \subseteq  F^{}_7$.
The subscript implicitly records (except for one instance) the number of times a theorem of Moret-Bailly
(Theorem~\ref{theorem:MB}) is applied. (This is not literally true, since many of the
references we invoke also appeal to variations of this theorem.)}
\end{remark}

As usual,  the abbreviations RAESDC  and RACSDC for an automorphic representation $\pi$ for
$\GL(n)$ stand for regular, algebraic, essentially-self-dual, and cuspidal; and
regular, algebraic, conjugate-self-dual, and cuspidal, respectively.

\medskip

\noindent  \bf Acknowledgements. \rm I would like to thank 
 Matthew Emerton and Richard Taylor for useful conversations,
 Jordan Ellenberg for a discussion about the inverse Galois problem, Toby Gee for explaining some
details of the proof of Theorem~\lifting of~\cite{BLGGT}, for keeping me informed of changes between the
first and subsequent versions of~\cite{BLGGT}, and for pointing out how 
Proposition~\ref{prop:res2} could be deduced immediately from~\cite{snow} (circumventing the previous slightly more
circuitous argument),
Mark Kisin for explaining  how his results in~\cite{KisinFM} could be used to deduce that every component of a certain
local deformation ring contained a global point,
  Brian Conrad for help in proving Lemma~\ref{lemma:conrad} and
  discussions regarding the material of Section~\ref{section:general}, and
  Florian Herzig for conversations about adequateness and the cohomology of
  Chevalley groups. Finally, I would like to thank the referee for a very thorough   reading
  of the manuscript and for many helpful comments.

\section{Local Deformation Rings}
Let $E$ be a finite extension of $\Q_p$ (the coefficient field), and let $V$ be a $d$-dimensional
vector space over $E$ with a continuous action of $G_{K}$, where
$K/\Q_p$ is a finite extension. 
Let us suppose that $V$ is potentially semi-stable~\cite{Fontaine}. 
Let
$$\tau: I_{K} \rightarrow \GL_d(\Qbar_p).$$
be a continuous representation of the inertia subgroup of $K$.
Fix an embedding $K \hookrightarrow \Qbar_p$. 
Attached to $V$ is a $d$-dimensional representation of the Weil--Deligne group
of $K$. 
If the restriction of this representation to the  inertia subgroup is equivalent
to $\tau$, we say that $V$ is of \emph{type $\tau$}.
Also associated to $V$ is a $p$-adic Hodge type $\v$, which records the
breaks in the Hodge filtration associated to $V$ considered as a de Rham
representation (cf.~\cite{Kisindef}, \S2.6).
Let $\F$ be a finite field of characteristic $p$, and let us now fix a representation
$$\rhobar: G_{K} \rightarrow \GL_n(\F).$$
Let $\Rbox_{\rhobar}$ be the universal framed deformation ring of $\rhobar$ considered as a $W(\F)$-algebra.
The following theorem is a result of Kisin (see~\cite{Kisindef}, Theorem 2.7.6).

\begin{theorem}[Kisin] There exists 
a quotient
$\Rboxtv_{\rhobar}$ of $\Rbox_{\rhobar}$ such that
the $\Qbar_p$-points of  the scheme $\Spec(\Rboxtv_{\rhobar}[1/p])$ are
exactly the $\Qbar_p$ points of $\Spec(\Rbox_{\rhobar})$ that
are potentially semi-stable of type $\tau$ 
and Hodge type
$\v$. It is unique if it is assumed to
be reduced and p-torsion free.
\end{theorem}

Note that restricting $\rhobar$ to some finite index subgroup $G_{L}$ induces
a functorial map of corresponding local deformation rings:
$$\Spec(\Rboxtv_{\rhobar}[1/p]) \rightarrow \Spec(\Rboxtv_{\kern+0.05em{\rhobar|G_L}}[1/p]),$$
where, by abuse of notation, $\tau$ in the second ring denotes the restriction of $\tau$ to $I_L$
(and correspondingly with $\mathbf{v}$).
We use $\mathds{1}$ to denote the trivial type.

\begin{df} A point of $\Spec(\Rboxtv_{\rhobar}[1/p])$ is \emph{very smooth} if it defines a smooth
point on $\Spec(\Rboxtv_{\kern+0.05em{\rhobar|G_L}}[1/p])$ for every finite extension $L/K$.
\end{df}

In sections~\S\lifta and ~\S\liftb of~\cite{BLGGT}, various notions of equivalence are defined
between representations. We would like to define a mild (obvious) extension of these definitions
when $v|p$.
Suppose that $\rho_1$ and $\rho_2$ are two continuous $d$-dimensional 
representations of $G_K$ with coefficients in some finite extension $E$ over $\Q_p$.
Let $\OL$ denote the ring of integers of $E$. Let us
 assume that $\rho_1$ and $\rho_2$ come with a specific integral structure, i.e., a given
$G_K$-invariant $\OL$-lattice. Equivalently, we may suppose that
$\rho_1$ and $\rho_2$ are representations $G_K \rightarrow \GL_d(\OL)$.
In particular, the mod-$p$ reductions $\rhobar_1$ and $\rhobar_2$ are well defined.
Such extra structure arises, for example, if the representations $\rho_i$ are the local representations
 attached to global representations whose mod-$p$ reductions are absolutely irreducible.

\begin{df}  Suppose $\rho_1$ and $\rho_2$ are two continuous $G_K$-representations
with given integral structure.  
  If $\rho_1$ and $\rho_2$ are potentially
 semi-stable, we say that $\rho_1 \zap \rho_2$ $($respectively, 
  $\rho_1 \rightsquigarrow \rho_2)$ if  $\rhobar:=\rhobar_1 \simeq \rhobar_2$, the representations
$\rho_1$ and $\rho_2$ have the same type $\tau$, the same Hodge type $\v$, and lie on the same irreducible
component of  $\Spec(\Rboxtv_{\rhobar}[1/p])$, and, furthermore, that $\rho_1$ corresponds to a 
very smooth point of
$\Spec(\Rboxtv_{\rhobar}[1/p])$ (respectively, smooth point
of $\Spec(\Rboxtv_{\rhobar}[1/p]))$.
\end{df}

\begin{remark} \emph{If $\rho_1 \zap \rho_2$, we say (following~\cite{BLGGT},~\S1.3, \S1.4) that $\rho_1$
\emph{very strongly connects} to $\rho_2$ (or $\rho_1$ ``zap'' $\rho_2$). (If $ \rho_1 \rightsquigarrow \rho_2$, then
$\rho_1$ \emph{strongly connects} to $\rho_2$, or $\rho_1$ ``squig''
 $\rho_2$.) If  $\rho_1 \zap \rho_2$, then clearly
 $\rho_1 \rightsquigarrow \rho_2$, and moreover  $\rho_1 |_{G_{L}} \zap \rho_2 |_{G_{L}}$ 
 (and hence $\rho_1 |_{G_{L}} \rightsquigarrow  \rho_2 |_{G_{L}}$) for any finite extension
  $L/K$.}
  \end{remark}
 
 \begin{remark} \emph{If $\rho_1$ and $\rho_2$ are both potentially crystalline representations,
then one may also consider the ring
$R^{\Box,\tau,\v,cr}_{\rhobar}$ parametrizing representations which are potentially
crystalline (cf.~\cite{Kisindef}). One may subsequently define the notions $\zap$ and
$\rightsquigarrow$ relative to this ring.  Since $\Spec(R^{\Box,\tau,\v,cr}_{\rhobar}[1/p])$ is smooth
(\emph{ibid.}),  the
relationships
$\rho_1 \zap \rho_2$ and $\rho_1 \rightsquigarrow \rho_2$ are  symmetric, and one  
may simply write
$\rho_1 \sim \rho_2$ (``$\rho_1$ connects to $\rho_2$'', cf.~\cite{BLGGT}). The scheme
$\Spec(R^{\Box,\tau,\v}_{\rhobar}[1/p])$ is not in general smooth,  so
we must impose strong connectedness in the potentially semi-stable case in order for the arguments 
of~\cite{BLGGT} to apply.
In this paper, whenever we have  $p$-adic representations $\rho_1$ and $\rho_2$, we shall
write $\rho_1 \sim \rho_2$ only when $\rho_1$ and $\rho_2$ are potentially crystalline, and
by writing this we mean that they are connected relative to 
$\Spec(R^{\Box,\tau,\v,cr}_{\rhobar}[1/p])$.}
\end{remark}

\medskip

In order to deduce in any particular circumstance 
that $\rho_1 \zap \rho_2$, it will be useful to have
some sort of criteria to determine when $\rho_1$ corresponds to a very smooth point.
If 
 $V$ is a potentially semi-stable representation of $G_K$,
 let $D = D_{\mathrm{pst}}(V)$ denote the corresponding weakly 
 admissible $(\varphi,N,\Gal(\overline{K}/K))$-module. Let
  $$D(k) \subset \Qbar_p \otimes_{\Q^{nr}_p} D$$
  denote the subspace generated by the (generalized) eigenvectors
  of Frobenius of slope $k$. 
  Since $N \varphi = p \varphi N$, there is a natural map
  $N: D(k+1) \rightarrow D(k)$.
  
\begin{lemma} \label{lemma:explicit} Let $V$ be a potentially semi-stable representation of $G_K$ of type $\tau$
and Hodge type $\v$.   \label{lemma:smooth}
 Suppose that $N: D(k+1) \rightarrow D(k)$
 is an isomorphism whenever the target and source are non-zero.
 Then $V$ is a very smooth point on  $\Spec(\Rboxtv_{\rhobar}[1/p])$.
 \end{lemma}

\begin{proof}
The explicit condition  follows from the proof of Lemma~3.1.5 of~\cite{Kisindef}.
\end{proof}

The condition of Lemma~\ref{lemma:smooth} is a (somewhat brutal) way of insisting
that the monodromy operator $N$ is as nontrivial as possible, given the action of
Frobenius. We note in passing that by Theorem~3.3.4 of~\cite{Kisindef}, $\Spec(\Rboxtv_{\rhobar}[1/p])$ admits a formally smooth dense open
subscheme.

\begin{example} \label{example:very} If $V$ is a $2$-dimensional representation
that is potentially semi-stable but not potentially crystalline, then
$\Sym^{n-1}(V)$  satisfies the conditions  of Lemma~\ref{lemma:explicit} and so is very smooth for all $n$.
\end{example}

\begin{remark} \emph{ One expects that a local
 $p$-adic representation associated to an RACDSC automorphic representation $\pi$
  is very smooth on the corresponding local deformation ring.
In fact, this would follow by the proof 
 of Lemma~\liftc of~\cite{BLGGT} if one had local--global compatibility at 
all primes (cf. Conjecture~1.1 and Theorem~1.2 of~\cite{TY}).
Since local--global compatibility is still unknown,  however, we must take more care in ensuring that the local representations associated
to automorphic representations we construct are (very) smooth.}
\end{remark}

\section{Realizing  local representations}

It will be useful in the sequel to quote the following extension of a theorem
of Moret--Bailly.

\begin{theorem} Let $E$ be a number field and let
$S$ be a finite set of places of $E$. Let $F/E$ be an auxiliary finite extension of number fields. 
Suppose that $X/E$ is a smooth geometrically connected variety. Suppose that:  \label{theorem:MB}
 For $v \in S$, $\Omega_v \subset X(E_v)$ is a non-empty open   $($for the $v$-topology$)$ subset.
Then there is a finite Galois extension $H/E$ and a point $P \in X(H)$ such that
\begin{enumerate}
\item $H/E$ is linearly disjoint from $F/E$.
\item Every place $v$ of $S$ splits completely in $H$, and if $w$ is a prime of $H$ above
$v$, there is an inclusion $P \in \Omega_v \subset X(H_w)$.
\item \label{item:descend}
Suppose that for any place $u$ of $\Q$, $S$ contains
either every place $v|u$ of $E$ or no such places.
Then one can choose $H$ to be a compositum $E M$ where:
\begin{enumerate}
\item $M/\Q$ is a totally real Galois  extension.
\item If there exists a $v \in S$ and a prime $p$ such that $v|p$, then $p$
splits completely in $M$.
\end{enumerate}
\end{enumerate}
\end{theorem}

\begin{proof}
Omitting part~(\ref{item:descend}), this is
 (a special case of) Proposition~2.1 of~\cite{TaylorSB}.
 To prove the additional statement, it suffices to apply
  Proposition~2.1 
of~\cite{TaylorSB}  to the restriction of scalars
  $Y = \mathrm{Res}_{E/\Q}(X)$. (Note that being smooth and geometrically connected is preserved
  under taking the restriction of scalars of a separable extension; for  see, for example, Theorem~A.5.9 of~\cite{PRG}.)
  \end{proof}

As a first application of this theorem, we prove the following result, which shows that the
inverse Galois problem  can be solved ``potentially,'' even with the imposition of local conditions
at a finite number of primes.

\begin{prop} Let $G$ be a finite group, let $E/\Q$ be a finite extension, and
$S$ a finite set of places of $E$. \label{prop:galois}
Let $F/E$ be an auxiliary finite extension of number fields.
For each finite place   $v \in S$, let  $H_v/E_v$ be a finite Galois
extension together with a fixed inclusion $\phi_v: \Gal(H_v/E_v) \rightarrow G$ with image $D_v$.
 For each
real infinite place $v \in S$, let $c_v \in G$ be an element of order dividing $2$.
There exists a number field $K/E$ and a finite Galois extension  of number fields $L/K$ 
with the following properties:
\begin{enumerate}
\item There is an isomorphism $\Gal(L/K) = G$.
\item $L/E$ is linearly disjoint from $F/E$.
\item All places in $S$ split completely in $K$. 
\item For all finite places $w$  of $K$ above $v \in S$, the local extension $L_w/K_w$ is equal to $H_v/E_v$.
Moreover, there is a commutative diagram:
$$
\begin{diagram}
\Gal(L_w/K_w) & \rTo & D_w \subset G \\
\dEquals & & \dEquals \\
\Gal(H_v/E_v) & \rTo^{\phi_v} & D_v \subset G \end{diagram}
$$
\item For all real places $w | \infty$ of $K$ above $v \in S$, complex conjugation $c_w \in G$ is conjugate
to $c_v$.
\end{enumerate}
\end{prop}

\begin{proof} Suppose that $G$ acts faithfully on $n$ letters, and
let $G \hookrightarrow \Sigma$ denote the corresponding map from $G$ to the symmetric group.
(Any group admits such a faithful action, e.g., the regular representation.)
There is an induced action of $G$  on
$\Q[x_1,x_2,\ldots,x_n]$, and we may  let $X_G = \Spec(\Q[x_1,x_2,\ldots,x_n]^G)$. 
There are corresponding morphisms
$$\mathbf{A}^n \rightarrow X_G \rightarrow X_{\Sigma}.$$
The scheme
$X_G$ is affine, irreducible, geometrically connected, and contains a Zariski dense smooth open subscheme.
The variety $X_{\Sigma}$
is  canonically isomorphic to affine space $\mathbf{A}^n$ over $\Spec(\Q)$ via the symmetric
polynomials. Under the projection to $X_{\Sigma}$, A $K$-point of $X_G$ (for any perfect
field $K$) gives a polynomial
over $K$ such that the Galois group of its splitting field $L$ is a (not necessarily transitive) subgroup of $G$. Without loss of generality, we may enlarge $S$ in the following way:
 For each conjugacy class $\langle g \rangle \in G$, we add to $S$ an auxiliary finite place $v$
and impose a local condition that the decomposition group at $v$ is unramified 
and is the subgroup generated by $g$.
 For all $v \in S$, let $\Omega_v \subset X_G(E_v)$ denote the set of points of $X_G$ defined as follows.
 If $v$ is a finite place, we suppose that
 the extension 
$L_v/E_v$ is isomorphic to $H_v/E_v$, and moreover, the corresponding action of $D_v$ on $G$ is the one given by $\phi_v$.
To construct such a point, we need to show that, given a Galois group $D_v = \Gal(H_v/E_v)$, any permutation representation
on $\Sigma$ 
can be realized by an $n$-tuple of elements of $H_v$. By induction, it suffices to consider the case
of transitive representations.  For a point $x \in \Sigma$ whose stabilizer
is supposed to be $C_v = \Gal(H_v/A_v)$,  choose $x$ to be a primitive element of $A_v = H^{C_v}_v$,
and then extend this choice to the orbit of $x$ under the abstract group $D_v$ via the Galois action arising from $\phi_v$.
If $v$ is an infinite place, we simply require that $L_v/E_v$ has Galois group $\langle c_v \rangle$.
By assumption, these sets are non-zero, and by Krasner's Lemma they are open.
We deduce by Theorem~\ref{theorem:MB} (applied to the smooth open subscheme of $X_G$)
 that there exists a Galois extension $L/K$ with Galois group
$H \subset G$ with the required local decomposition groups at each place $w$ above $v$.
By construction, for every $g \in G$ there exists a finite unramified place $w$ in $K$ such
that the conjugacy class of Frobenius at $w$ in $\Gal(L/K)$ is the conjugacy class of $g$.
It follows that the intersection of $H$ with every conjugacy class of $G$ is nontrivial, and hence
$H = G$ by a well known theorem of Jordan (see Theorem~$4'$ of~\cite{Serre}). Thus the theorem is established.
\end{proof}

\begin{remark} \emph{A weaker version of Proposition~\ref{prop:galois}, namely, that every finite group
$G$ occurs as the Galois group $\Gal(L/K)$ for some extension of number fields, is a trivial consequence
of the fact that $\Sigma = S_n$ occurs as the Galois group of some extension of $\Q$, since we may take
$L$ to be any such  extension and $K = L^{G}$. If we insist that some place $v$ splits completely in $K$, however,
this will typically force $L$ to also split completely at $v$.}
\end{remark}

Let $\rho: G_{F^+} \rightarrow \GL_2(\Qbar_p)$ be as in Theorem~\ref{theorem:even2}.
After increasing $F^{+}$ (if necessary), we may assume that $\rho$ is semi-stable
at all primes of residue characteristic different from $p$. 
Attached to $\rho$ is a residual representation
$\rhobar:G_{F^{+}} \rightarrow \GL_2(\F)$ for some finite field $\F$ of characteristic $p$.

\begin{prop} \label{prop:res} There exists a  totally real field $F^{+}_1/F^{+}$ 
and a residual Galois representation $\rbar_{\residual}:
G_{F^{+}_1} \rightarrow \GL_2(\F)$
with the following properties:
\begin{enumerate} 
\item All primes above $p$ split completely in $F^{+}_1$.
\item The residual representation $\rbar_{\residual}:G_{F^{+}_1} \rightarrow \GL_2(\F)$ has image
containing $\SL_2(\F_p)$. \label{part2}
 \item For each $v|p$ in $F^{+}$, and for each $w$ above $v$ in $F^{+}_1$,
 there is an isomorphism $\rbar_{\residual} |D_w \simeq \rhobar| D_{v}$.
\item If $v \nmid p$, then
$\rbar_{\residual}|D_{v}$ is unramified.
\item $\rbar_{\residual}$ is totally odd at every real place of $F^{+}_1$.
\item $F^{+}_1 \cap \Q(\zeta_{p}) = \Q$ and $F^{+}_1 \cap \Q(\ker(\rhobar)) = \Q$.
\end{enumerate}
\end{prop}

\begin{proof} Proposition~\ref{prop:galois}  immediately guarantees a residual
representation satisfying all the conditions with the possible exception of $(4)$,
which can be achieved by a further base extension.
\end{proof}

\begin{prop} \label{prop:res2} There exists a totally real field
$F^{+}_2/F^{+}_1$ and a Hilbert modular form $f$ for
$F^{+}_2$ 
with a corresponding residual representation
$\rhobar_f: G_{F^{+}_2} \rightarrow \GL_2(\F)$ with the following properties:
\begin{enumerate}
\item There is an isomorphism
$\rhobar_f \simeq \rbar_{\residual}|G_{F^{+}_2}$.
\item All primes above $p$ split completely in $F^{+}_2$.
\item $F^{+}_2 \cap \Q(\zeta_{p}) = \Q$, $F^{+}_2 \cap \Q(\ker(\rhobar)) = \Q$, and $F^{+}_2 \cap \Q(\ker(\rbar_{\residual})) = \Q$. 
\end{enumerate}
\end{prop}

\begin{proof} This follows immediately from Theorem~5.1.1 of~\cite{snow}, taking into account
Proposition~8.2.1 of \emph{ibid}. It suffices to note (in the notation of \emph{ibid}.) that there always exists
at least one definite type $t(v)$ compatible with $\rbar_{\residual}|D_v$ for each $v|p$, but this is exactly Proposition~7.8.1 of~\cite{snow}.
\end{proof}

\begin{remark} \label{remark:oddcase}
\emph{If $\rho$ is \emph{odd} for all infinite places of $F^{+}_1$, then we may take
$\rbar_{\residual}$ to be $\rhobar$, and Proposion~\ref{prop:res2} implies that
$\rhobar |G_{F^{+}_2}$ is modular.
By Theorem~2.2.18 of~\cite{Kisin}, it follows in this case
that  $\rho$ is modular over $F^{+}_2$.
(This is not literally correct, because $\rhobar$ may not 
have image containing $\SL_2(\F_p)$ and so not virtually satisfy
Condition~\ref{part2} of Proposition~\ref{prop:res}. However, one may
check that the only fact used about the image of $\rbar$ so far is that
it is irreducible.)}
\end{remark}

\medskip

Having realized the representations $\rhobar|D_v$ for $v|p$ inside
the mod-$p$ reduction of some  Hilbert modular form $f$, 
we now realize the representations
$\rho|D_v$ in characteristic zero
as coming from Hilbert modular forms
(to the extent that it is possible). 

\begin{prop}   \label{prop:kisin}
There exists a  Hilbert modular form $g$ over
$F^{+}_2$
with the following properties:
\begin{enumerate} 
\item The residual representation $\rhobar_g:G_{F^{+}_2} \rightarrow \GL_2(\F)$ 
is equal to $\rhobar_f$.
\item For each place $w|p$ of  $F^{+}$ and $v|w$ of $F^{+}_2$,
$\rho_g |D_v \sim \rho|D_w$ if $\rho |D_w$ is potentially crystalline, and
 $ \rho_g |D_v  \zap \rho |D_w$ otherwise.
\item For each finite place $v$ of $F^{+}_2$ away from $p$, $\rho_g |D_v$
is unramified.
\end{enumerate}
\end{prop}

\begin{proof} Consider the modular representation $\rhobar_f = \rbar_{\residual}|G_{F^{+}_2}$ constructed in
Proposition~\ref{prop:res2}. By construction, it is modular of minimal level
and is unramified outside $p$.
It follows from Theorem~2.2.18 and Corollary~2.2.17 of~\cite{KisinFM} that (in the notation of \emph{ibid.})
$M_{\infty}$ is faithful as an $\bar{R}_{\infty}$-module. 
To orient the reader, $M_{\infty}$ is a module which is built (by Taylor--Wiles patching) from a sequence
of faithful Hecke modules of finite level, and  $\bar{R}_{\infty}$ is a (patched) decorated universal deformation ring
which encodes deformations of a prescribed type at all $v|p$. In particular, the ring $\bar{R}_{\infty}$
detects \emph{all} the components of the local deformation rings at $v|p$ of the given type. 
Hence, at least morally (since at this point we are working with patched modules) the faithfulness of $M_{\infty}$
says that ``every component'' of the corresponding local deformation rings is realized globally. To show this 
formally, we need to pass from the patched level back down to finite level by setting
all the
auxiliary Taylor--Wiles variables $\Delta_{\infty}$ equal to zero.
Explicitly,  the Taylor--Wiles--Kisin method yeilds and isomorphism
$\bar{R}_{\infty} =\bar{R}^{\Box,\psi}_{\Sigma_p}
\llbracket x_1, \ldots, x_g \rrbracket$ of $\bar{R}_{\infty}$ as  a power series ring
over a tensor product of local deformation rings. 
Here $\Sigma_p$ denotes the set of places dividing $p$.
 Consider a component
$Z$ of $\Spec(\bar{R}_{\infty})$ such that the characteristic zero points of $Z$ lie on the
same local component as $\rho|D_v$ for $v | p$ (if $\rho|D_v$ lies on multiple components,
choose any component).
Since all the (equivalent) conditions of Lemma~2.2.11 of~\emph{ibid.}\ hold,
we know (as in the proof of and notation of that lemma) that
$\bar{R}_{\infty}$ is a finite torsion free
$\OL\llbracket \Delta_{\infty} \rrbracket$-module. In particular, $Z$ surjects onto
$\Spec(\OL \llbracket \Delta_{\infty}\rrbracket)$. In particular, there is a non-trivial
fibre at $0$. Since
 $M_0 = M_{\infty} \otimes_{\OL\llbracket \Delta_{\infty} \rrbracket} \OL$ is a space
 of \emph{classical} modular forms (of minimal level), we deduce the existence of $g$.
 Note that if $\rho|D_w$ is potentially semi-stable but not potentially crystalline, then $\rho_g$ is
 very smooth by Example~\ref{example:very}, and hence
 $\rho_g |D_v \zap \rho |D_w$ for $v|p$.
\end{proof}

\medskip

\begin{remark} \emph{The faithfulness of $M_{\infty}$ as a  $\bar{R}_{\infty}$-module
is not a clear consequence of the Fontaine--Mazur conjecture. 
That is, \emph{a priori}, the collection of all global representations may
surreptitiously conspire to avoid a given local component. Thus, without any further
ideas, the new cases of the Fontaine--Mazur conjecture proved by Emerton~\cite{EmertonFM} do
not allow us to realize all local representations globally when
$$\rhobar|D_p \sim \left( \begin{matrix} \epsp & * \\ 0 & 1 \end{matrix} \right).$$}
\end{remark}

\medskip

We have now constructed a Hilbert modular form $g$ whose $p$-adic representation
is a ``shadow'' of $\rho$, that is,  lies on the same
component as $\rho$ of every local deformation space at a place dividing $p$.
However, the global mod-$p$
representations $\rhobar$ and $\rhobar_g$ are unrelated.
In order to prove a modularity statement, we will need to construct
a second pair of shadow representations with the same residual
representation as $\rhobar$ and $\rhobar_g$. It will never be the case,
however, that $\rhobar$ will be automorphic over a totally real field unless
$\rhobar$ is totally odd. 
The main idea of~\cite{CalegariFM} was to  consider the
representation
$\Sym^2(\rhobar)$ as a conjugate self-dual representation over some CM field.
In the sequel, we shall construct a pair of RACSDC ordinary crystalline shadow representations
which realize the mod-$p$ representations $\Sym^2(\rhobar)$ and
$\Sym^2(\rhobar_g)$.
By abuse of notation, let $\v$  denote the Hodge type of $\rho$ 
for any prime dividing $p$,
so $\v$ is literally a collection of Hodge types for each $v|p$ in $F^{+}$. Similarly,
suppose that $\rho$ is potentially semi-stable of type $\tau$, where $\tau$ refers to a collection
of types for all $v|p$
 (adding subscripts
would add nothing to the readability of the following argument).

\begin{prop} 
Let $p > 7$ be  prime. \label{prop:ord}
There exists a totally real field $F^{+}_5/F^{+}_2$,
a CM extension  $F^{}_5/F^{+}_2$, and a RACDSC automorphic representation $\pi$ over $F^{}_5$
with a  corresponding Galois representation
$\rho_{\pi}:  G_{F_5} \rightarrow \GL_3(\Qbar_p)$ such that:
\label{lemma:first}
\begin{enumerate}
\item $\rho_{\pi}$ is unramified at all places  not dividing $p\cdot \infty$.
\item For every $v|p$,   $\rho_{\pi} |G_v$ is ordinary and crystalline with Hodge type $\Sym^2 \w$, where
$\w$ is a Hodge type of some $2$-dimensional de Rham representation.
\item For all $v|p$ and for all $i$, there is an inequality
$\dim \mathrm{gr}^i (\Sym^2 \v \otimes \Sym^2 \w) \le 1$. 
\item The image of the restriction of $\rhobar$ to $G_{F^{+}_5}$ is the
image of $\rhobar$ on $G_{\Q}$.
\item The restriction of $\rhobar_g$ to $G_{F^{+}_5}$ has image containing $\SL_2(\F_p)$.
\item The residual Galois representation $\rhobar_{\pi}:G_{F^{}_5} \rightarrow \GL_3(\F)$ is isomorphic
to the  restriction of $\Sym^2(\rhobar)$ to $G_{F^{}_5}$.
\item The Hilbert modular form $g$ remains modular over $F^{+}_5$.
\item The compatible family of Galois representations associated to $\pi$ is \label{part8}
irreducible after restriction to any finite index subgroup of $G_{F_5}$.
\item  If $\rho_g |D_p$ is not potentially crystalline, the representation $\rho_{\pi} \otimes \Sym^2(\rho_g) |D_p$ is a
very smooth point of 
$$\Spec(R^{\Box,\Sym^2(\tau) \otimes \Sym^2(\mathds{1}), \Sym^2 \v \otimes \Sym^2 \w}[1/p]).$$
\item $F^{}_5 \cap \Q(\zeta_p) = \Q$.
\item  $\rhobar_{\pi}|D_v$ is trivial for every $v|p$.
\end{enumerate}
\end{prop}

\begin{proof} Let $F^{+}_3/F^{+}_2$ be a totally real field for which
$\Sym^2(\rhobar)$ becomes completely trivial at all $v|p$. (In general, the field $F^{+}_3$ will
be highly ramified at $p$). Increasing $F^{+}_3$ if necessary, assume that
the restriction of $\Sym^2(\rhobar)$ is unramified outside $v|p$,  and that
there exists a CM extension $F^{}_3/F^{+}_3$ which is totally unramified at all finite places.
By Proposition~3.3.1 of~\cite{BLGGT}, 
$\Sym^2(\rhobar)$ 
 admits minimal ordinary automorphic lifts
over some CM extension $F^{}_4 = F^{+}_4 . F^{}_3$. (Here by minimal we simply mean unramified outside $p$ and
ordinary and crystalline at $v|p$).
(We use the fact that $\rhobar$ is not of dihedral type,
so $\Sym^2(\rhobar)$ is irreducible, and that $p \ge 2(3 + 1)$.)
 The resulting automorphic representation $\pi$ satisfies
condition $(1)$.
Note that $F^{}_4$ may be chosen to be disjoint from any auxiliary field.
This implies that we may construct $F^{}_4$ so that conditions  $(1)$, $(4)$, $(5)$,  $(6)$, $(10)$,
and $(11)$ hold.
To preserve the modularity of $g$, let $l$ denote an auxiliary prime which is totally split in $F^{}_4$, and which admits a prime
$\lambda | l$ in $F^{+}_2$ such that:
\begin{enumerate}
\item $\rho_{g,\lambda}$ is crystalline in the Fontaine--Laffaile range, and the difference of any two weights is at most $l - 2$.
\item $\rhobar_{g,\lambda}$ is absolutely irreducible.
\end{enumerate}
We may now apply Theorem~4.5.1 of~\cite{BLGGT} to deduce the existence of an extension
$F^{}_5 = F^{+}_5 \cdot F^{}_4$ of $F^{}_4$ so that $\pi$ and $g$ \emph{simultaneously} remain automorphic
over $F^{}_5$ (by base change, if $g$ is automorphic over $F^{}_5$, it is also so over $F^{+}_5$).
Note that the key condition to check in order to apply Theorem~4.5.1 of ~\cite{BLGGT} is potential diagonalizability. This
is easy by construction:  the $p$-adic Galois
representations associated to $\pi$ are ordinary for all $v|p$, and the $\lambda$-adic Galois representations associated
to $g$ are crystalline in the Fontaine--Laffaile range, and $\lambda | l$ is unramified in $F^{}_4$.
Thus we can ensure that
condition $(7)$ holds.
It is also easy to see that $F^{}_5$ can be chosen to preserve the conditions that we have already established above.
Note that the amount freedom allowed in choosing the
Hodge type is equivalent to the freedom to choose
$\mu$ in Proposition~3.2.1 of~\cite{BLGGT}, and it is easy to find a choice of weight such that $\pi$ to
satisfy conditions $(2)$ and $(3)$.  In particular, one could  let $\w$ be any regular Hodge
type which is independent of the choice of embedding $F^{}_5 \rightarrow \Qbar_p$ and for which
there are sufficiently large gaps between non-zero terms of $\mathrm{gr}^i(\w)$.
 To verify that the compatible system associated to $\pi$ is  irreducible over
 any finite index subgroup of $G_{F_5}$ 
 (and so verifies condition $(8)$),
we invoke Theorem~2.2.1 of~\cite{BR}. It suffices to note that $\rhobar_{\pi}$
has non-solvable image, and thus $\pi$ is not induced from an algebraic Hecke
character over a solvable extension.
Finally, we must show that
$\pi$ can be chosen to satisfy $(9)$.
The slopes of Frobenius of an ordinary crystalline representation are given
by the breaks in the Hodge filtration. In particular, we may choose a $\w$ so 
that the integers $i$ such that  $\mathrm{gr}^i \Sym^2(\w) \ne 0$ each differ by
$\ge 4$. If $\Sym^2(\rho_g)$ is potentially semistable but not potentially crystalline,
it follows from Lemma~\ref{lemma:smooth} that for such a $\w$ that
\emph{any} such tensor product $\rho_{\pi} \otimes \Sym^2(\rho_g)$ will be very smooth.
Thus, choosing $\w$ appropriately,  very smoothness is automatically satisfied.
 \end{proof}

\medskip

We may now construct a Hilbert modular
form $h$ as follows.

\begin{prop} Let $p > 5$. There exists a totally real field $F^{+}_6/F^{+}_5$
and a Hilbert modular form $h$ over   $F^{+}_6$ with a corresponding Galois representation  $\rho_h: G_{F^{+}_6} \rightarrow \GL_2(\Qbar_p)$  such that:
\label{prop:second}
\begin{enumerate}
\item For every $v|p$, $\rho_h |D_v$ is ordinary and crystalline with Hodge type $\w$.
\item The residual representation $\rhobar_{h}:G_{F^{+}_6} \rightarrow \GL_2(\F)$ is isomorphic to the restriction of
$\rhobar_g$. 
\item The images of  $\rhobar$ and $\rhobar_{g}$ remain unchanged upon
restriction to $G_{F^{+}_6}$.
\item The Hilbert modular form $g$ remains modular over $F^{+}_6$, and the
RACDSC representation $\pi$ remains modular over $F^{}_6= F^{+}_6 . F^{}_5$.
\item For all places $v$ not dividing $p$,
$\rho_{h} |D_v \sim \rho |D_v$.
\item $F^{}_6 \cap \Q(\zeta_p) = \Q$.
\end{enumerate}
\end{prop}

\begin{proof} We may prove this by modifying the proof of Proposition~\ref{lemma:first} as follows.
Modify the field $F^{+}_3$ so that $\rhobar_{g}$ is ordinary at all $v|p$ (but still has large image),
and such that $\rhobar_{g} |D_v$ admits a  modular lift which is ramified
and semi-stable lift for those primes $v \nmid p$ which ramify
in $\rho$, ordinary for $v|p$, and unramified everywhere else. 
The existence of such a lift follows from Theorem~6.1.9 of~\cite{BLGGtwo}, at least in the weight $0$ case --- the existence of a lift
in weight $\w$ then follows from Hida theory.
 Then let $F^{+}_6$
 denote a field
for which (using Theorem~4.5.1 of~\cite{BLGGT} again) we can simultaneously establish the modularity of this
lift (which will correspond to $h$), as well as the modularity of the $p$-adic and $\lambda$-adic Galois
representations associated to $\pi$ and $g$  respectively.
 (It may have been  more consistent to
have combined  Propositions~\ref{prop:ord} and~\ref{prop:second} into a single Lemma, but
it would have been more unwieldy.)
\end{proof}

\section{The proof of Theorem~\ref{theorem:even2}}

Let $L/\Q$ denote a field which contains the coefficient field of $g$ and $\pi$.
Let $(L,\rho_{g,\lambda})$ denote the compatible family of Galois representations associated to $g$.
There exists a prime of $\OL_F$ dividing $p$ such that the corresponding mod-$p$
representation is  $\rhobar_g$, which, by construction, has non-solvable image. It follows
that the form $g$ does not have complex
multiplication, and hence the images of $\rho_{g,\lambda}$ for all $\lambda$ contains an open subgroup
of $\SL_2(\Z_{l})$ where $\lambda | l$. 
 
\begin{prop} There exists a totally real field $F^{+}_7/F^{+}_6$, a quadratic CM extension $F^{}_7/F^{+}_7$
and a RACDSC automorphic representation $\Pi$ of $\GL(9)/F^{}_7$ such that:
\begin{enumerate}
\item The compatible family of Galois representations associated to $\Pi$ is the restriction
to $G_{F^{}_7}$ of the  family
$$(L,\Sym^2(\rho_{g,\lambda}) \otimes \rho_{\pi,\lambda}).$$
\item The images of  $\rhobar$ and $\rhobar_{h}$ remain unchanged upon
restriction to $G_{F^{}_7}$.
\end{enumerate}
\end{prop}

\begin{proof}  The
 compatible system is essentially self-dual, orthogonal (automatically since $n = 9$ is odd), and  has distinct Hodge--Tate weights 
(by assumption $3$ of Lemma~\ref{lemma:first}). Let us verify that it is 
irreducible. Since $g$ does not have CM (as $\rhobar_{g}$ is not dihedral), the Zariski closure of $\Sym^2(\rho_{g,\lambda})$ is $\PGL(2)$.
Since $\rho_{\pi,\lambda}$ is irreducible, It follows that the tensor product will be irreducible unless
$\rho_{\pi,\lambda}$ is  a twist of $\Sym^2(\rho_{g,\lambda})$.
By multiplicity one~\cite{J} for $\GL(3)$, we deduce that $\Sym^2(g)$ is a twist of
$\pi$. This contradicts the fact that the  mod-$p$ residual
representations representations $\Sym^2(\rhobar_{g})$
and $\rhobar_{\pi}$ are not twists of each other, since one extends to a totally odd representation
($\Sym^2 \rbar_{\residual}$) 
of $G_{F^{+}_1}$ and the other to a representation ($\Sym^2(\rhobar)$) of $G_{F^{+}_1}$ which is even at some infinite place.
The potential automorphy follows from Theorem~A of~\cite{BLGGT} (as follows from the proof of that theorem, the
field $F^{+}_7$ can be chosen to be disjoint from any finite auxiliary field, establishing condition $(2)$).
\end{proof}

\medskip

Let us now write $E^{+} = F^{+}_7$ and $E = F^{}_7$, and consider the
representations $\rho$, $\rho_g$, $\rho_{\pi}$, and $\rho_{h}$ as representations
of $G_{E}$. 
Without loss of generality, we may assume that $\rho$ is even for at least
one real place of $F^{+}$ (and hence also of $E^{+}$).
Let us consider the representation
$$\vrho: \Sym^2(\rho) \otimes \Sym^2(\rho_h): G_{E} \rightarrow \GL_9(\Qbar_p).$$
By construction, we observe that
$\vrhobar = \Sym^2(\rhobar) \otimes \Sym^2(\rhobar_h) = \rhobar_{\pi} \otimes \Sym^2(\rhobar_g) = \rhobar({\Pi})$ 
is residually modular. 
Moreover, we find that
$$ \rho(\Pi)|D_v =  \rho_{\pi} \otimes \Sym^2(\rho_g) | D_v  \rightsquigarrow   \Sym^2(\rho) \otimes \Sym^2(\rho_h) |D_v
= \vrho |D_v$$
for all $v$, with the possible exception of $v|p$.
 By
 Lemma~3.4.3 of~Geraghty~\cite{G}, the ordinary deformation rings are smooth and connected
 (by Lemma~\ref{prop:ord}, $(11)$, 
 $\rhobar_{\pi}|D_v$ is trivial for $v|p$), 
 and hence, for $v|p$, 
 $$ \rho_{\pi}|D_v \sim \Sym^2(\rho_h) |D_v.$$
 On the other hand, by construction (Proposition~\ref{prop:kisin} $(2)$), we also have (for $v|p$)  that
 $$\Sym^2(\rho_g) |D_v \zap \Sym^2(\rho) |D_v.$$
If $\rho|D_v$ is potentially crystalline, then all four representations are potentially
crystalline at $v$, and
we deduce that 
$$\rho_{\pi} \otimes \Sym^2(\rho_g)|D_v \sim \Sym^2(\rho) \otimes \Sym^2(\rho_h) |D_v.$$
On the other hand, if $\rho|D_v$ is not potentially crystalline, then neither is
$\rho_g|D_v$, and we deduce from condition $(9)$ of Lemma~\ref{lemma:smooth} 
that the left hand side corresponds to a very smooth point of the corresponding
local deformation ring $\Spec(R^{\Box,\Sym^2(\tau) \otimes \Sym^2(\mathds{1}), \Sym^2 \v \otimes \Sym^2 \w}[1/p])$.
Hence
$$\rho(\Pi)|D_v = 
\rho_{\pi} \otimes \Sym^2(\rho_g)|D_v \zap \Sym^2(\rho) \otimes \Sym^2(\rho_h)|D_v
= \vrho|D_v,$$
and thus $\rho(\Pi)|D_v \rightsquigarrow \vrho|D_v$.
We now prove that $\vrhobar$ has adequate image (in the sense of~\cite{Thorne}).
Recall that the fixed fields corresponding to $\Sym^2(\rhobar)$ and $\Sym^2(\rhobar_h)$ are
disjoint by construction.
The images of $\Sym^2(\rhobar_h)$ is either $\PSL_2(\F)$ or
$\PGL_2(\F)$, by construction. Since $\PSL_2(\F)$ is simple,  by Goursat's Lemma, the image of
$\vrhobar$ contains (with index at most two) the direct product of the image of $\Sym^2(\rhobar)$
and $\PSL_2(\F)$. 
 Since both $\Sym^2(\rhobar)$
and $\PSL_2(\F)$ have adequate image, their direct product satisfies the cohomological condition
of adequateness, since $H^1(\Gamma \times \Gamma',V \otimes V') = 0$ whenever
$H^1(\Gamma,W) = H^1(\Gamma',W') = 0$ by inflation restriction
(the irreducible constitutents of $\Sym^2(V \otimes V')$ are of the form $W \otimes W'$ where $W$ and $W'$
are irreducible constituents of $\Sym^2(W)$ and $\Sym^2(W')$ respectively).
Since the image of $\vrhobar$ contains this product with index dividing two, its image
still satisfies the cohomological condition of adequateness, by inflation-restriction (and
the fact that $p \ne 2$).
 We deduce that $\vrhobar$ is adequate
 by Lemma~2(ii) of~\cite{Herzig}.
It follows from Theorem~\ref{theorem:semistable} below
(c.f.
 Theorem~\liftingflat  of~\cite{BLGGT})
that
$\vrho$
  is modular over $E$. 
(Since $n = 9$ is odd, the slightly regular condition is vacuous.)
Since the Galois representations $\rho$ and $\rho_h$ extend to the totally real subfield $E^{+}$,
so does the representation $\vrho$, and hence  (by~\cite{AC}) $\vrho$ comes from
a RAESDC representation for $\GL(9)/E^{+}$. By assumption, however, there exists a real
place of $E^{+}$ and a correponding complex conjugation $c_v \in G_{E^{+}}$ such that
$\rho(c_v)$ is a scalar. It follows that the image $\vrho(c_v)$ of $c_v$ 
is conjugate to
$$\left(\begin{matrix} 1 & & \\ & 1 & \\ & & 1 \end{matrix} \right) \otimes
\left(\begin{matrix} 1 & & \\ & -1 & \\ & & 1 \end{matrix} \right),$$
which has trace $+3$. Since the representation  $\vrho$ is irreducible,
 automorphic,  and $9$ is odd, this contradicts the main theorem (Proposition~A) of~\cite{TP}, which says that such
representations must be ``odd'' in the sense that the trace of complex conjugation must be $\pm 1$.
This completes the proof of Theorem~\ref{theorem:even2}.

\medskip

\begin{remark} \emph{In~\cite{CalegariFM}, we proved (Theorem~1.3)
that for an imaginary quadratic field $K$, 
ordinary representations $\rho: G_K \rightarrow \GL_2(\Qbar_p)$ (satisfying
certain supplementary hypotheses) had parallel weight. One may ask
whether the methods of this paper can be used to generalize that result
(presuming that $p$ splits in $K$). Starting with the tensor representation
$\psi = \rho \otimes \rho^c$, one is lead to a $16$ dimensional representation
$\vrho = \psi \otimes \rho_f \otimes \rho_g$ for non-CM Hilbert modular forms
$f$ and $g$ which one may show is  modular for $\GL(16)$ over some
totally real field $E^{+}$. It is not apparent, however, how this
might lead to a contradiction, since $\vrho(c)$ has trace $0$ for any
complex conjugation $c$, and known cases of functoriality do not
yet allow one to deduce the modularity
of $\psi$ from the modularity of $\vrho$. Another approach is to consider
the representation $\psi = \Sym^2 \rho \otimes \Sym^2 \rho^c$, and a corresponding
$81$ dimensional representation
$\vrho = \psi \otimes \Sym^2 \rho_f \otimes \Sym^2 \rho_g$. In this case, one
should be able to deduce that $\vrho$ is modular for $\GL(81)/E^{+}$ for some totally
real field $E^{+}$, and that $\vrho(c)$ is conjugate to
{\small
$$ \left( \begin{matrix}
1 & 0 & 0 & & & & & \\
0 & 1 & 0 & & & & & \\
0 & 0 & 1 & & & & & \\
& & & 0 & 1 & & & & \\
& & & 1 & 0 & & & & \\
& & & & & 0 & 1 & & \\
& & & & & 1 & 0 & & \\
& & & & & & & 0 & 1 \\
& & & & & & & 1 & 0 \end{matrix} \right)
\otimes \left(\begin{matrix} 1 & & \\ & -1 & \\ & & 1 \end{matrix} \right)
\otimes \left(\begin{matrix} 1 & & \\ & -1 & \\ & & 1 \end{matrix} \right),$$
}
which has trace $3$, contradicting the main theorem of~\cite{TP}.
However, this approach only works when $\psi$ has distinct Hodge--Tate
weights, and correspondingly one may only deduce (if $\rho$ has
Hodge--Tate weights $[0,m]$ and $[0,n]$ with $n \ge m$ positive) that
either $n = m$ (which is the expected conclusion) \emph{or} $n = 2m$.}
\end{remark}

\section{Applications to Universal Deformation Rings}
\label{section:even}

If $F/\Q$ is a number field,
and $\rhobar:G_F \rightarrow \GL_n(\F)$ is an irreducible continuous Galois representation, let
$\X = \X_S(\rhobar)$ denote the rigid analytic space corresponding to the universal deformation ring of $\rhobar$
unramified outside a finite set of primes $S$.
The space $\X$ contains a (presumably countable) set of points which are de Rham ($=$ potentially semi-stable)
at all places above $p$.
We call such deformations \emph{geometric}, and denote the corresponding set of points
by $\Xgeom$.
Gouv\^{e}a  and Mazur, using a beautiful construction they called the infinite fern,
showed (under the assumption that $\rhobar$ was unobstructed) that when $F = \Q$, $n = 2$, and $\rhobar$ is odd that
$\Xgeom \subset \X$ is  Zariski dense (see~\cite{GM}). This raises the general question of when
$\Xgeom$ is Zariski dense in $\X$. The work of Gouv\^{e}a and Mazur has been extended
(in the same setting) by others, in particular B\"{o}ckle~\cite{Bo}
(cf.~Theorem~1.2.3 of~\cite{EmertonFM}). Chenevier~\cite{Chen} has recently
generalized the infinite fern argument to apply to certain conjugate self dual representations
for $n \ge 3$ over CM fields, and has shown that the Zariski closure of $\Xgeom$ is (in some precise sense)
quite large.
As a consequence of our main result,  however, we prove the following theorem.

\begin{theorem} Let $p > 7$, let $F^{+}/\Q$ be a totally real field in which $p$ splits completely, \label{theorem:empty}
and let $\rhobar: G_{F^{+}} \rightarrow \GL_2(\F)$ be a continuous irreducible representation
whose image contains $\SL_2(\F_p)$.
Suppose that $\rhobar$ is even for at least one real place of $F^{+}$.
Suppose that  for all $v|p$, 
$$\rhobar|{I_v} \sim \left(\begin{matrix}  \psi_v & * \\ 0 &  1 \end{matrix} \right)$$
where $\sim$ denotes up to twist, and $\psi_v \ne \epsp$ is assumed to have order $> 2$, and $* \ne 0$. 
Let $S$ be any finite set of places of $F^{+}$.
Then $\Xgeom$ is empty, that is, 
$\rhobar$ has no geometric deformations.
\end{theorem}

\begin{proof} Assume otherwise. Let $\rho$ be a point of $\Xgeom$.
By Theorem~\ref{theorem:even2}, there exists at least one $v|p$ such that the Hodge--Tate
weights of $\rho$ at $v$ are equal. To be potentially semi-stable (or even Hodge--Tate)
of parallel weight zero is equivalent to being unramified over a finite extension.
Thus, up to twist, $\rho | I_v$  has  finite image, and, in particular, the projective image of
$\rho | I_v$ is finite.
The only finite subgroups of
$\PGL_2(\Qbar_p) \simeq \PGL_2(\C)$ are either cyclic, dihedral, $A_4$, $S_4$, or $A_5$. Thus, the projective image of $\rhobar |I_v$ must be one of these groups.
By assumption, the projective
image of $\rhobar|I_v$ is a non-dihedral group of order divisible by $p > 7$, hence $\rho|I_v$ can not be finite up to twist either,
and  $\rho$ does not exist.
\end{proof}

We have the following corollary:

\begin{corr} There exist  absolutely irreducible representations $\rhobar$ such that
the subset $\Xgeom \subset \X$ is \emph{not} Zariski dense. In fact,  there exist
representations such that 
$\Xgeom$ is empty, but $\X$ has arbitrary large dimension.
\end{corr}

\begin{proof} Let $F^{+}/\Q$ be a totally real field, and let $\rhobar:G_{F^{+}} \rightarrow
\GL_2(\F)$ be a continuous irreducible representation satisfying the conditions
of Theorem~\ref{theorem:empty}. 
Suppose, furthermore, that $\psi_v \ne \epsp^{-1}$
for any $v|p$. The existence of  such representations is guaranteed by
Proposition~\ref{prop:galois}. More precisely, start with an auxiliary totally real field $E^{+}$ of degree $> 1$,
and fix an infinite place $v$ of $E$.
Then use Proposition~\ref{prop:galois} to construct $\rhobar$ which 
are odd for all $w|\infty$ in $F^{+}$ with $w \nmid v$, and even for all $w|v$. By Theorem~\ref{theorem:empty}, $\Xgeom = \emptyset$
for any finite set of auxiliary primes $S$. 
Using the strategy of the proof of  Theorem~1(a) of~\cite{R},
one deduces the existence of sets $S$ for which
there exists a family of Galois representations
$$\rho: G_{F^{+},S} \rightarrow \GL_2(W(\F)[[T_1, \ldots, T_n]])$$
of (relative over $W(\F)$) dimension $n = \delta + 2 r$, where $\delta$ is the Leopoldt defect and
$r$ is the number of infinite places of $F^{+}$ at which
$\rhobar$ is odd, and such that every specialization of $\rho$ is
\emph{reducible} at $D_v$ for all $v|p$. (More generally, one may replace $F^{+}$
by any number field $F$ and obtain a family of dimension 
$\delta + 2s + 2r$, where $r$ is the number of real places at which $\rhobar$ is odd, and $s$
is the number of complex places. (If one fixes the determinant, this decreases
to $s + 2r$.) The special case when $F$ is an imaginary quadratic field 
is Lemma~7.6 of~\cite{CM}, but the proof for an arbitrary number field is essentially
the same.
By construction, $2r = [F^{+}:\Q] - [F^{+}:E^{+}]$. In particular, we can make $r$ ---
 and thus $\X$ --- arbitrarily large.
\end{proof}

Similarly, we note the following cases of the Fontaine--Mazur which do not require
any assumption on the Hodge--Tate weights or the parity of $\rhobar$:

\begin{corr} Let
$\rho: G_{\Q} \rightarrow \GL_2(\Qbar_p)$ be a continuous 
Galois representation which is unramified except at a finite number of primes. Suppose
that $p > 7$, and, furthermore, that
\begin{enumerate}
\item $\rho |D_p$ is potentially semi-stable.
\item The residual representation $\rhobar$ is absolutely
irreducible and is not of dihedral type.
\item $\rhobar|D_p$ is of the form:
$\displaystyle{ \left( \begin{matrix} \psi_1 & * \\ 0 & \psi_2 \end{matrix} \right)}$ where:
\begin{enumerate}
\item $*$ is ramified.
\item $\psi_1/\psi_2 \ne \epsp$, and $\psi_1/\psi_2 |I_p$ has order $> 2$.
\end{enumerate}
\end{enumerate}
Then $\rho$ is modular.
\end{corr}

Although  proposition~\ref{prop:galois} guarantees the existence of infinitely
many even  Galois
representations over totally real fields with image containing $\SL_2(\F_p)$,
it may also be of interest to construct at least one example over $\Q$ (with $p \ge 11$).
We shall do this now.

\begin{lemma} Let $K/\Q$ be a degree $11$ extension with splitting field $L/\Q$ such that:
\begin{enumerate}
\item $11$ is totally ramified in $K$.
\item $G = \Gal(L/\Q) = \PSL_2(\F_{11})$.
\item $\ord_{11}(\Delta_{K/\Q}) \not\equiv 0 \mod 10$.
\end{enumerate}
Let $\p$ denote a prime above $p: = 11$ in $L$, and let $I \subseteq D \subset G$
denote the corresponding inertia and decomposition groups. \label{lemma:biglocal}
Then $I = D$ has order $55$ and is the full Borel subgroup of $G$.
\end{lemma}

\begin{proof} Since $11$ is totally ramified in $K/\Q$, the inertia group $I$ has order divisible
by $[K:\Q] = 11$. Since $I \subset D$ is solvable, it follows that $D$ is contained inside a Borel subgroup of $G$.
Let $F$ and $E$ denote the images of $L$ and $K$ under their embedding into $\Qbar_p$
corresponding to $\p$. We note that $D = \Gal(F/\Q_p)$, and we have the following diagram:
$$
\begin{diagram}
L & \rLine & K & \rLine & \Q \\
\dInto & & \dInto & & \dInto \\
F & \rLine_{e \kern-0.1em{f} = 1,5} & E & \rLine_{e = 11} & \Q_p \\
\end{diagram}
$$
 It suffices to assume that $|I| = 11$ and deduce a contradiction.
  Suppose that $|D| = 11$, so $F = E$ is abelian over $\Q_p$.
   By local class field theory, $F/\Q_p$ is (up to an unramified twist)
   given by the degree $p = 11$ subfield of the $p^2$-roots of unity.
Thus
$$\Delta_{E/\Q_p} = \Delta_{F/\Q_p} =  11^{20},$$
and thus $\ord_{11}(\Delta_{E/\Q_p}) \equiv  0 \mod 10$, a contradiction.

Suppose that $|D| = 55$ and $|I| = 11$. 
Let $I_n \subseteq I$ denote the lower ramification groups.
The $p$-adic valuation of the discriminant of $F/\Q_p$ is given by
the following formula:
$$\ord_{p}(\Delta_{F/\Q_p}) = \frac{|D|}{|I|} \cdot \sum_{n=0}^{\infty} \left( |I_n| - 1\right).$$
By assumption, $|D|/|I| = 5$ and $|I_n| = 11$ or $1$ for all $n$. We deduce that
$\ord_{p}(\Delta_{F/\Q_p}) \equiv 0 \mod 50$. On the other hand,
$$\Delta_{F/\Q_p} = N_{F/E}(\Delta_{F/E}) \cdot (\Delta_{E/\Q})^{5}.$$
Since $F/E$ is unramified, we deduce that
$$\ord_{11}(\Delta_{E/\Q_p}) = 
\frac{1}{5} \ord_{11} (\Delta_{F/\Q_p}) \equiv 0 \mod 10.$$
Since $\ord_{11}(\Delta_{K/\Q}) = \ord_{11}(\Delta_{E/\Q_p})$, the lemma follows.
\end{proof}

\begin{corr} \label{corr:better} 
There exists a surjective even  representation $\rhobar: G_{\Q} \rightarrow \SL_2(\F_{11})$
with no geometric deformations.
\end{corr}

\begin{proof} Let $K$ be the field obtained by adjoining to $\Q$ a root of the irreducible
polynomial
{\small
\begin{align*}
x^{11} & + 154 \cdot x^{10} + 8591 \cdot x^9 + 207724 \cdot x^8 + 1846031 \cdot x^7 - 2270598 \cdot x^6 - 63850600 \cdot x^5 \\
& +  73646034 \cdot x^4 + 582246423 \cdot x^3 - 1610954576 \cdot x^2 + 1500989952 \cdot x - 481890304.
\end{align*}
}
This polynomial was obtained by specializing the parameters $a$ and $t$ of a polynomial
found by Malle (Theorem~9.1 of~\cite{Malle}) to $a = 14$ and $t = -419$
respectively.
One may verify that $11$ is totally ramified in $K/\Q$,  that
the splitting field $L/\Q$ is totally real with Galois group $G = \PSL_2(\F_{11})$,
and that  the discriminant has a prime factorization as follows:
$$\Delta_{K/\Q} = 11^{12} \cdot 133462088669841218191^{4}.$$
Since  $\ord_{11}(\Delta_{K/\Q}) = 12$, it follows from
Lemma~\ref{lemma:biglocal} that the inertia group $D$ at $11$
is the full Borel subgroup of $G$. We note also the factorization
$$133462088669841218191 \cdot \OL_K = \p_1 \p^2_2 \p^2_3 \q^2_1 \q_2,$$
where $\p_i$ and $\q_i$  have residue degree $1$ and $2$ respectively.
It follows that the residue degree and the ramification index of every  prime above $133462088669841218191$
in $L$ is even (in fact, $2$).
Since  $133462088669841218191 \equiv 3 \mod 4$,  it follows from a theorem
of B\"{o}ge (Theorem~1.1 of~\cite{Kluners}) that $L$
embeds in a $\SL_2(\F_{11})$-extension $N/\Q$. The action  of complex conjugation on this extension is, by construction,
either trivial or by $\left(\begin{matrix} -1 & 0 \\ 0 & -1 \end{matrix} \right)$. In the former case, $N$ is totally real. In the latter case,
we may twist by some (any) quadratic character of an imaginary quadratic field, and the corresponding field extension (which we still
call $N$) will be totally real.
Let $\rhobar: G_{\Q} \rightarrow \Gal(N/\Q) = \SL_2(\F_{11})$ denote the corresponding
representation. We now show that  $\rhobar|D_{11}$ satisfies the conditions
of Theorem~\ref{theorem:empty}. Since the decomposition group at $11$ maps surjectively
onto the Borel of $\PSL_2(\F_{11})$, it is contained in the Borel of $\SL_2(\F_{11})$.
Any such representation may be twisted (in $\GL_2(\F_{11})$) to be of the form
$$\left( \begin{matrix} \psi & * \\ 0 & 1 \end{matrix} \right)$$ for some character $\psi$.
If $\psi$ has order two, then the image of $D_{11}$ will not surject onto the Borel of
$\PSL_2(\F_{11})$ (twisting does not affect this projection). If 
$\psi = \epsp$, however, then $ \det(\rhobar) = \epsp \cdot \chi^2$ for
some character $\chi$.  Yet $\rhobar$ has image in $\SL_2(\F_{11})$ and thus has trivial
determinant,  whilst $\epsp$
is not the square of any character. Thus $\psi \ne \epsp$, and Theorem~\ref{theorem:empty}
applies.
\end{proof}

\section{Some remarks on the condition $p > 7$}

One may wonder if the condition that $p > 7$ is used in an
essential way in this argument. At the very least, one will
require that the representation
$$\Sym^2 \rhobar: G_{\Q} \rightarrow \GL_3(\F_p)$$
be adequate. This fails to have adequate image if
the image of $\rhobar$ is $\SL_2(\F_p)$ and $p \le 7$. 
The author expects that
for $p = 7$ it should be sufficient to assume that
the projective image of $\rhobar$ is either $A_4$, $S_4$, $A_5$
or contains $\PSL_2(\F_{49})$, that for $p = 5$ that projective
image  is either $A_4$, $S_4$, or contains $\PSL_2(\F_{25})$, and
that for $p = 3$ the image contains $\PSL_2(\F_{27})$.
The main technical issue to address is exactly what form
of adequateness is required in Proposition~3.2.1 
of~\cite{BLGGT}, although another issue is that many
of the references we cite include assumptions on
$p$ which would also need to be modified 
(using~\cite{Thorne}). The methods (in principle) also apply
with $p = 2$, although many more technical ingredients
would need to be generalized in this case, in particular, the work
of~\cite{KisinFM}.

\section{A remark on potential modularity theorems}
\label{section:general}

In recent modularity lifting results~\cite{BLGHT,BLGGT,BLGG,BLGGtwo,G} for $l$-adic
representations, a weak form
of local-global compatibility at primes $v|l$  is invoked (in this section only, we work
with $l$-adic representations rather than $p$-adic representations in order
to be most compatible with~\cite{BLGHT}), namely, that automorphic
forms of level co-prime to $l$  give rise to crystalline Galois representations (of the
correct weight).
 In general, local--global compatibility 
for RACDSC cuspidal forms for $\GL(n)$ is
only known in the crystalline case (as follows from~\cite{TY}), although partial results are known
in the semi-stable case.
 In this section, we show how to prove modularity
 results similar to Theorem~\lifting of~\cite{BLGHT}  without local--global
 compatibility, allowing for a modularity lifting theorem
 in the potentially semi-stable case. We claim no great originality, 
 as the proof is essentially the same
as the proof
 of  Theorem~\lifting of~\cite{BLGGT} (or Theorem~7.1 of~\cite{Thorne}) with the addition
 of one simple ingredient (Lemma~\ref{lemma:conrad} below). (The authors of~\cite{BLGGT} inform
 me that they have a different method for dealing with the
 potentially semi-stable case which was not included in~\cite{BLGGT}
 for space reasons.)
 (One should also compare the statement of this theorem to Theorem~7.1 of~\cite{Thorne}.)

 \begin{theorem} Let $F$ \label{theorem:semistable}
 be an imaginary CM field with maximal
  totally real subfield  $F^{+}$. Suppose $l$ is odd and let
  $n$ be a positive integer. Let
  $$r:G_F \rightarrow \GL_n(\Qbar_{l})$$
  be a continuous representation and let
  $\rbar$ denote the corresponding residual representation. Also, let
  $$\mu: G_{F^{+}} \rightarrow \Qbar^{\times}$$
  be a continuous homomorphism. Suppose that
  $(r,\mu)$ enjoys the following properties:
  \begin{enumerate}
  \item $r^c \simeq r^{\vee} \eps^{1-n}_{l} \mu |_{G_F}$.
  \item $\mu(c_v)$ is independent of $v|\infty$.
  \item the reduction $\rbar$ is absolutely irreducible and
  $\rbar(G_{F(\zeta_{l})}) \subset \GL_n(\Fbar_{l})$ is
  adequate.
  \item There is a RAECDSC automorphic representation
  $(\pi,\chi)$ of
  $\GL_n(\mathbf{A}_F)$ with the following
  properties.
  \begin{enumerate}
  \item $(\rbar,\mubar) \simeq
  (\rbar_{l,\iota}(\pi),\rbar_{l,\iota}(\chi))$.
  \item For all
  places $v \nmid l$ of $F$ at which
  $\pi$ or $r$ is ramified, we have
  $$r_{l,\iota}(\pi) |_{G_{F_v}}
  \sim r|_{G_{F_v}}.$$
  \item
  For all places $v | l$ of $F$, $r|_{G_{F_v}}$ is potentially semi-stable and we have
  $$r_{l,\iota}(\pi) |_{G_{F_v}}
  \rightsquigarrow r |_{G_{F_v}}.$$
\item If $n$ is even, $\pi$ has slightly regular weight.
  \end{enumerate}
  \end{enumerate}
  Then $(r,\mu)$ is automorphic. 
  \end{theorem}

\begin{remark} \emph{The only difference between this theorem and
Theorem~\lifting of~\cite{BLGGT} is that:
\begin{enumerate}
\item We do \emph{not} assume that $\pi$ is potentially
unramified above $l$.
\item We require for $v|l$ that $r_{l,\iota}(\pi) |_{G_{F_v}}
  \rightsquigarrow r |_{G_{F_v}}$ rather than
  $r_{l,\iota}(\pi) |_{G_{F_v}}
  \sim r |_{G_{F_v}}$.
\item We impose that $\pi$ has slightly regular weight (this
is only a condition when $n$ is even). This is because we
require that the Galois representation associated to $\pi$ can
be realized geometrically. Perhaps using the methods
of~\cite{Chenevier} this assumption can be eliminated.
Alternatively, one could try to work with the representation
$r_{l,\iota}(\pi)^{\otimes 2}$ which can be realized
geometrically (see~\cite{Caraiani}).
  \end{enumerate}
}
  \end{remark}

\begin{proof} We make the following minor adjustment to the
proof of Theorem~7.1 of~\cite{Thorne} (cf.
Theorem~\lifting of~\cite{BLGGT}).
The character $\chi$ may be untwisted after some
solvable ramified extension.
We now modify the deformation problem consided in the proof
of Theorem~3.6.1 of~\cite{BLGG} as follows.
At $v|l$, we consider deformations (with a fixed finite
collection of Hodge  types $\v$) that
become potentially semi-stable over some fixed extension $L/K$,
where $L$ will be determined below.
The argument proceeds in the same manner 
 \emph{providing} there exists  a map from $R \rightarrow \mathbf{T}$,
 and, (in light of the fact that the deformation rings at $l$ might have non-smooth
 points in characteristic zero) utilizing   the hypothesis
 that  $r_{l,\iota}(\pi) |_{G_{F_v}}$ corresponds
to a smooth point in the local deformation ring.
For any fixed level structure, the Galois
 representations arising from quotients of $\mathbf{T}_{Q_n}$ are
 potentially semi-stable over some extension $L/K$ by a theorem
 of Tsuji~\cite{Ts}. However, to define the appropriate ring $R$, we must ensure
 that the corresponding types of these Galois representations lie in some finite
 set \emph{independent} of the set of auxiliary primes $Q_n$. 
In particular,   we are required to show that we may find
 a \emph{fixed} $L/K$ such that the Galois representations obtained
 by adding any set of auxiliary Taylor--Wiles primes are  semi-stable over 
 the same field $L$.

 \begin{lemma} Let $K/\Q_l$ be a finite extension, \label{lemma:conrad}
 and let $X$ be a proper flat
 scheme over $\Spec(\OL_K)$ with smooth generic fibre.
 Then there exists a finite
 extension $L/K$ with the following property: For every finite \'{e}tale map
 $\pi: Y \rightarrow X$, the \'{e}tale cohomology groups
 $H^i(Y_{\overline{K}},\Qbar_l)$ become semi-stable as representations
 of $G_L$.
 \end{lemma}

 \begin{proof} 
 After making a finite
 extension $L/K$, there
 exists (via the theory of alterations~\cite{dJ}) a proper hypercovering
 $X_{\bullet}$ of $X$ 
 such that  for all $n \le 2 \dim(X)$:
 \begin{enumerate}
 \item $X_n$ is proper and flat over $\Spec(\OL_L)$
 \item $X_n$ has smooth generic fibre and semi-stable  special fibre.
 \end{enumerate}
 By cohomological descent, there is a spectral sequence
 $$H^m(X_{n,\overline{L}},\Qbar_l) \Rightarrow H^{m+n}(X_{\overline{L}},\Qbar_l).$$
The cohomology groups on the left are semi-stable by Tsuji's proof of $C_{\mathrm{st}}$. 
Since the property of being semi-stable is preserved by taking
sub-quotients,
 it follows that the $G_L$-representation $H^i(X_{\overline{L}},\Qbar_l)$
(for $i \le 2 \dim(X)$)  has
an exhaustive filtration by semi-stable $G_L$-modules.  Hence
the semi-simplification of $H^i(X_{\overline{L}},\Qbar_l)$ is semi-stable.
Since $H^i(X_{\overline{L}},\Qbar_l)$ 
is also de Rham~\cite{Faltings}, it follows
that $H^i(X_{\overline{K}},\Qbar_l) = 
H^i(X_{\overline{L}},\Qbar_l)$ is itself semi-stable as a
$G_L$-representation for
$i \le 2 \dim(X)$. The cohomology of $X_{\Kbar}$ vanishes outside this range, so the
claim follows for all $i$.
 This recovers Tsuji's Theorem $C_{\mathrm{pst}}$ (Indeed, this is essentially Tsuji's argument).
 Let us now consider a finite \'{e}tale morphism $Y \rightarrow X$.
We may form a hypercovering $Y_{\bullet} = X_{\bullet} \times_X Y$ of $Y$.
 The properties $(1)$ and $(2)$ of the hypercovering $X_{\bullet}$ are preserved 
 under base change by a finite \'{e}tale map, and thus  the cohomology of $Y$ is also semi-stable over $L$.
 \end{proof}

\begin{remark}
\emph{For an expositional account of the theory of hypercoverings
and cohomological descent 
in the \etale topology, see~\cite{Conrad}.}
\end{remark}
 
Consider the (compact) Shimura variety $\mathrm{Sh}$ (over $\Spec(\OL_K)$)
associated to the unitary similitude group $G$ as in~\cite{Book,HT,Shin}, 
where $\LL = \LL_{\xi}$ is a lisse sheaf correpsonding to the automorphic vector bundle for
an irreducible algebraic
representation $\xi$ of $G$.
Let $\mathcal{A}^m$ denote the $m$th self-product of the
universal abelian variety over $\mathrm{Sh}$, and
let $\pi: \A^m \rightarrow \Sh$ denote the (smooth, proper) projection. For a
suitable $m$,
one can write $\LL = e R^m \pi_{*} \Qbar_l(\t)$ for some $m =
m_{\xi}$ and
$\t = \t_{\xi}$,
and $e$ is some idempotent (cf.~\cite{HT}, p.98).
Finally, let $\Sh(N)$ denote the finite \etale cover of $\Sh$ corresponding to
the addition of auxiliary level $N$-structure for some $N$ co-prime to $p$.
Let $\A^m(N)$ denote the base change of $\A^m$ to $\Sh(N)$; it is finite
\etale over $\A^m$. 
The Leray spectral sequence gives a map
$$H^{p}(\mathrm{Sh}(N)_{\Kbar},R^q \pi_{*} \Qbar_l(\t))
 \Rightarrow H^{p+q}(\mathcal{A}^m(N)_{\Kbar},\Qbar_l(\t)).$$
Multiplication by $n$ on $\A$ induces the map $n^j$
on $R^j 
\pi_* \Qbar_l$. The formation of the spectral sequence
is compatible with this map, and hence it
commutes with the differentials in the spectral sequence, which
correspondingly degenerates (cf. the argument
of Deligne, p.169 of~\cite{Deligne}). Thus
$$H^n(\Sh(N)_{\Kbar},\LL) = e H^{n}(\mathrm{Sh}(N)_{\Kbar},R^m \pi_{*} \Qbar_l(\t))$$
 occurs as a subquotient of
$H^i(\mathcal{A}(N)_{\Kbar},\Qbar_l(\t))$ for some $i$. 
Let $X = \A^m$ and $Y = \A^m(N)$. 
By Lemma~\ref{lemma:conrad}, we deduce that
$$H^{n}(\mathrm{Sh}(N)_{\Kbar},\LL)$$ 
 is  semi-stable over a fixed extension $L/K$ for all $N$ depending only
on $m$ and $\xi$.
If $\pi$ has slightly regular weight, the the Galois representation
associated to $\pi$ in~\cite{Shin} can be realized geometrically
in the \etale cohomology of an automorphic sheaf on $\Sh$ as
considered above.
Moreover, the Galois representations corresponding to automorphic
forms arising in the Taylor--Wiles constructions at auxiliary
primes arise in the \etale cohomology of the same sheaf on
$\Sh(N)$ for some auxiliary level $N$.
 It follows that the local
Galois representations associated to the Hecke rings
$\mathbf{T}_{Q_n}$ are all quotients of a local deformation ring
 involving a fixed finite set of types, which is the necessary local
input for the modularity lifing theorem  (Theorem~7.1) of~\cite{Thorne}.
  \end{proof}

\bibliographystyle{amsalpha}
\bibliography{pst}

\end{document}